\documentclass[10pt,a4paper]{amsart}

\usepackage{amssymb}
\usepackage{amsmath}
\usepackage{amscd}

\begin{document}

\theoremstyle{plain}
\newtheorem{theorem}{Theorem}[section]
\newtheorem{proposition}[theorem]{Proposition}
\newtheorem{lemma}[theorem]{Lemma}
\newtheorem{corollary}[theorem]{Corollary}
\newtheorem{conj}[theorem]{Conjecture}

\theoremstyle{definition}
\newtheorem{definition}[theorem]{Definition}
\newtheorem{exam}[theorem]{Example}
\newtheorem{remark}[theorem]{Remark}

\numberwithin{equation}{section}

\title[$m$-quasi-Einstein metrics on simple Lie groups]
{Non-trivial $m$-quasi-Einstein metrics on simple Lie groups}

\author{Zhiqi Chen}
\address{School of Mathematical Sciences and LPMC \\ Nankai University \\ Tianjin 300071, P.R. China} \email{chenzhiqi@nankai.edu.cn}

\author{Ke Liang}
\address{School of Mathematical Sciences and LPMC \\ Nankai University \\ Tianjin 300071, P.R. China}

\author{Fuhai Zhu}
\address{School of Mathematical Sciences and LPMC \\ Nankai University \\ Tianjin 300071, P.R. China}



\begin{abstract}
We call a metric $m$-quasi-Einstein if $Ric_X^m$, which replaces a gradient of a smooth function $f$ by a vector
field $X$ in $m$-Bakry-Emery Ricci tensor, is a constant multiple of the metric tensor. It is a generalization of Einstein
metrics which contains Ricci solitons. In this paper, we focus on left-invariant metrics on simple Lie groups. First, we prove that $X$ is a left-invariant Killing vector field if the metric on a compact simple Lie group is $m$-quasi-Einstein. Then we show that every compact simple Lie group admits non-trivial $m$-quasi-Einstein metrics except $SU(3)$, $E_8$ and $G_2$, and most of them admit infinitely many metrics. Naturally, the study on $m$-quasi-Einstein metrics can be extended to pseudo-Riemannian case. And we prove that every compact simple Lie group admits non-trivial $m$-quasi-Einstein Lorentzian metrics and most of them admit infinitely many metrics. Finally, we prove that some non-compact simple Lie groups admit infinitely many non-trivial $m$-quasi-Einstein Lorentzian metrics.
\end{abstract}

\maketitle


\setcounter{section}{0}
\section{Introduction}
A natural extension of the Ricci tensor to smooth metric measure spaces is the $m$-Bakry-Emery Ricci tensor
\begin{equation}
Ric^m_f=Ric+\nabla^2f-\frac{1}{m}df\otimes df
\end{equation}
where $0<m\leqslant \infty$, $f$ is a smooth function on $M^n$, and $\nabla^2f$ stands for the Hessian form. Instead of a gradient of a smooth function $f$ by a vector
field $X$, $m$-Bakry-Emery Ricci tensor was extended by Barros and Ribeiro Jr in \cite{BR1} and
Limoncu in \cite{L1} for an arbitrary vector field $X$ on $M^n$ as follows:
\begin{equation}
Ric^m_X=Ric+\frac{1}{2}{\mathfrak L}_Xg-\frac{1}{m}X^*\otimes X^*
\end{equation}
where ${\mathfrak L}_Xg$ denotes the Lie derivative on $M^n$ and $X^*$ denotes the canonical 1-form
associated to $X$. With this setting $(M^n,g)$ is called an $m$-quasi-Einstein metric, if there exist a vector
field $X\in {\mathfrak X}(M^n)$ and constants $m$ and $\lambda$ such that
\begin{equation}\label{qEin}
Ric^m_X=\lambda g.\end{equation}

An $m$-quasi-Einstein metric is called trivial when $X\equiv 0$. The triviality definition is equivalent to saying that $M^n$ is an Einstein manifold.
When $m=\infty$, the equation (1.3) reduces to a Ricci soliton, for more details see \cite{Ca1} and the references therein. Following the terminology of Ricci soliton, an $m$-quasi-Einstein metric is called expanding, steady or shrinking, respectively, if $\lambda<0$, $\lambda=0$ or $\lambda>0$. When $m$ is a positive integer and $X$ is a gradient vector field, it corresponds to a warped product Einstein metric, for more details see \cite{CSW1}. Classically the study on $m$-quasi-Einstein are considered when $X$ is a gradient of a smooth function $f$ on $M^n$, see \cite{An1,AK1,CSW1,Cor1,ELM1,KK1}.

It is well known that compact homogeneous Ricci solitions are Einstein. It can be deduced from the work of Petersen-Wylie \cite{PW1} where homogeneous gradient Ricci solitions are studied and a result of Perelman \cite{Pere1} which states that compact Ricci solitions are gradient. In \cite{Jab1}, Jablonski gives a new proof of this result. Furthermore, Jablonski proves in \cite{Jab2} that the generalized Alekseevskii conjecture for Ricci solitons is equivalent to the Alekseevskii conjecture for Einstein metrics. By the study about $m$-quasi-Einstein metrics on 3-dimensional homogeneous manifolds, it is pointed out in \cite{BRJ1} that not every compact $m$-quasi-Einstein metric is gradient. That is, compact homogeneous $m$-quasi-Einstein metrics are not necessarily Einstein for $m$ finite. Moreover, these examples in \cite{BRJ1} show that Theorem 4.6 of \cite{HPW1} can not be extended for a non-gradient vector field.

For the examples given in \cite{BRJ1}, $m$-quasi-Einstein metrics are finitely many. In this paper, we focus on left-invariant metrics on simple Lie groups and obtain infinitely many $m$-quasi-Einstein metrics among them.

As has been shown, the result that compact homogeneous Ricci solitions are Einstein fails for $m$ finite. On the other hand, the result is equivalent with that $X$ is a Killing vector field. First, we prove the same result for left-invariant $m$-quasi-Einstein metrics on compact Lie groups. That is,
\begin{theorem}\label{killingfield}
Let $G$ be a compact Lie group with a left-invariant metric $\langle,\rangle$. If $X$ is a vector field on $G$ such that $\langle,\rangle$ is $m$-quasi-Einstein, i.e., $Ric^m_X=\lambda\langle,\rangle$, then $X$ is a left-invariant Killing vector field.
\end{theorem}

It is shown in \cite{CSW1} that every compact gradient $m$-quasi-Einstein metric with constant scalar curvature is trivial. Thus $X$ is non-gradient if $\langle,\rangle$ is non-trivial.

Secondly, based on basis facts on Levi-Civita connections and Ricci curvatures with respect to naturally reductive metrics on compact simple Lie groups, we prove that $X$ belongs to the center of ${\mathfrak k}$ if the naturally reductive metric is $m$-quasi-Einstein. Furthermore, we prove the following result.
\begin{theorem}\label{main}
Every compact simple Lie groups except $SU(3)$, $E_8$ and $G_2$ admits non-trivial $m$-quasi-Einstein metrics which are naturally reductive. In particular, every compact simple Lie group except $SU(3)$, $E_8$, $F_4$ and $G_2$ admit infinitely many non-trivial $m$-quasi-Einstein metrics.
\end{theorem}

There are some remarks on Theorem~\ref{main}.
\begin{enumerate}
    \item It is shown in \cite{DZ1} that every compact simple Lie group except $SO(3)$ admits naturally reductive Einstein metrics. But the problem how many Einstein metrics there are on compact simple Lie groups is still an open problem. Here we get infinitely many $m$-quasi-Einstein metrics on most of compact simple Lie groups.
    \item We obtain non-trivial $m$-quasi-Einstein metrics on ${\mathfrak su}(2)$ with $\lambda\leqslant 0$. It means that the result in \cite{KK1} for a gradient $m$-quasi-Einstein metric can't be extended to a non-gradient case. Furthermore, by the examples on ${\mathfrak su}(2)$ with $\lambda<0$, we can construct non-compact and non-solvable Lie groups admitting expanding $m$-quasi-Einstein metrics, which shows that the Alekseevskii conjecture for $m$-quasi-Einstein metrics fails.
\end{enumerate}

Next, we study pseudo-Riemannian metrics instead of Riemannian metrics and prove the following result on $m$-quasi-Einstein Lorentzian metrics.
\begin{theorem}\label{main2}
Every compact simple Lie group admits non-trivial $m$-quasi-Einstein Lorentzian metrics. In particular, every compact simple Lie group except $E_8$, $F_4$ and $G_2$ admit infinitely many non-trivial $m$-quasi-Einstein metrics.
\end{theorem}

It is necessary to point out that the above result for Einstein Lorentzian metrics is still open.

Finally, we construct an $m$-quasi-Einstein pseudo-Riemannian metric on some non-compact simple Lie group by an $m$-quasi-Einstein pseudo-Riemannian metric on a compact simple Lie group, and show the following theorem.
\begin{theorem}\label{main3}
There are infinite many non-trivial $m$-quasi-Einstein Lorentzian metrics on the non-compact simple Lie groups $SO^*(2k)$ for $k\geqslant 3$, $SP(k,\mathbb{R})$ for $k\geqslant 2$, $SO_0(2,k)$ for $k\geqslant 3$, $SU(1,k)$ for $k\geqslant 3$, $SU(k_1,k_2)$ for $k_1,k_2\geqslant 2$, $e_6^{-14}$ and $e_7^{-25}$.
\end{theorem}

\section{The proof of Theorem~\ref{killingfield}}
Let $G$ be a compact Lie group with the Lie algebra ${\mathfrak g}$, let $B$ be the Killing form of $\mathfrak g$, and let $\langle,\rangle$ be a left-invariant metric on $G$, i.e., $$\langle \nabla_XY,Z\rangle+\langle Y,\nabla_XZ\rangle=0$$ for any $X,Y,Z\in{\mathfrak g}$. There is a linear map $D$ of $\mathfrak g$ satisfying $$\langle X,Y\rangle=(-B)(X,D(Y)),\forall X,Y\in {\mathfrak g}.$$ It follows that $D$ is symmetric with respect to $B$. Hence there exists a basis $\{X_1,\cdots,X_n\}$ of $\mathfrak g$ with respect to $-B$ such that $D(X_i)=\lambda_i(X_i)$. Moreover, $$\langle X_i,X_j\rangle=\delta_{ij}\lambda_i.$$ For any vector field $X$ on $G$, $X=\sum_{i=1}^nf_iX_i$, where every $f_i$ is a smooth function on $G$. With respect to the left-invariant $\langle,\rangle$,
\begin{eqnarray*}
   Ric^m_X(X_i,X_i)&=&Ric(X_i,X_i)+\langle \nabla_{X_i}X,X_i\rangle-\frac{1}{m}\langle X,X_i\rangle^2 \\
                   &=&Ric(X_i,X_i)+X_if_i+\sum_{j=1}^nf_j\langle \nabla_{X_i}X_j,X_i\rangle-\frac{1}{m}f_i^2\\
                   &=&Ric(X_i,X_i)+X_if_i+\sum_{j=1}^nf_j\langle [X_i,X_j],X_i\rangle-\frac{1}{m}f_i^2\\
                   &=&Ric(X_i,X_i)+X_if_i-\lambda_i\sum_{j=1}^nf_j B([X_i,X_j],X_i)-\frac{1}{m}f_i^2\\
                   &=&Ric(X_i,X_i)+X_if_i-\frac{1}{m}f_i^2.
\end{eqnarray*}
Furthermore in addition $Ric^m_X=\lambda\langle,\rangle$, i.e., $\langle,\rangle$ is $m$-quasi-Einstein. Then we have $X_if_i-\frac{1}{m}f_i^2=a_i$ for some constant $a_i$. 
Let $x\in G$ such that $f_i(x)$ is maximal, we have $(X_if_i)(x)=0$. Clearly $f_i(xy)$ is maximal at the point $y=e\in G$, so we can assume that $x=e$ without loss of generalization. Hence $f_i^2(e)=-ma_i$. Let $a_i=-\frac{b_i^2}{m}$ for some $b_i\geqslant 0$. Then we have $f_i(e)=\pm b_i$. Assume that $x_1\in G$ satisfies that $f_i(x_1)$ is minimal. Similarly, $f_i(x_1)=\pm b_i$, then $|f_i(y)|\leqslant b_i$ for any $y\in G$. So if $f_i(e)=-b_i$, then $f_i\equiv -b_i$; if $f_i(x_1)=b_i$, then $f_i\equiv b_i$. In the following, assume that $f_i(e)=b_i$. Clearly $X_if_i\leqslant 0$. It follows that $f_i(\exp tX_i)\equiv b_i$ for $t\leqslant 0$. Let $T=\{\overline{\exp tX_i}|t\leqslant0\}$, which is equivalent with $\{\overline{\exp tX_i}|t\in{\mathbb R}\}$. Then $f_i(y)=b_i$ for any $y\in T$. For any $z\in G$, let $z_1\in zT$ be the point such that $f_i(z_1)$ is maximal on $zT$. Then $f_i(z_1\exp tX_i)$ is maximal at $t=0$. Similarly, $f_i(zT)=f_i(z_1T)\equiv b_i$. In particular, $f_i(z)=b_i$. Since $z$ is arbitrary, we have:
\begin{theorem}\label{leftinv}
Let $G$ be a compact Lie group with a left-invariant metric $\langle,\rangle$. If $X$ is a vector field on $G$ such that $Ric_X^m=\lambda\langle,\rangle$, i.e., $\langle,\rangle$ is $m$-quasi-Einstein, then $X$ is left-invariant.
\end{theorem}

Considering a Lie group $G$ as a special homogeneous manifold, we have the following facts and results. Let $M$ denote the set of left-invariant metrics on $G$. For any left-invariant metric $Q$ on $G$, the tangent space $T_QM$ at $Q$ is left-invariant symmetric, bilinear forms on $\mathfrak g$. Define a Riemannian metric on $M$ by
$$(v,w)_Q=\mathrm{tr}\ vw=\sum_iv(e_i, e_i)w(e_i, e_i),$$
where $v,w\in T_QM$ and $\{e_i\}$ is a $Q$-orthonormal basis of $\mathfrak g$.

\begin{lemma}\label{sc}
Let $G$ be a unimodular Lie group with a left-invariant metric $\langle,\rangle$. Given $Q\in M$, denote by $ric_Q$
and $sc_Q$ the Ricci and scalar curvatures of $(G,Q)$, respectively. The
gradient of the function $sc: M\rightarrow{\mathbb R}$ is
$(\mathrm{grad}\ sc)_Q =-ric_Q$
relative to the above Riemannian metric on $M$.
\end{lemma}

For a proof of this result for homogeneous manifolds, see \cite{He1} or \cite{Ni1}. Assume that $\langle,\rangle$ is $m$-quasi-Einstein. That is, there exists a vector field $X$ on a unimodular Lie group $G$ such that $Ric+\frac{1}{2}{\mathfrak L}_X\langle,\rangle-\frac{1}{m}X^*\otimes X^*=\lambda\langle,\rangle$. In addition, assume that $X$ is a left-invariant vector field, i.e., $X\in {\mathfrak g}$. Since $\langle,\rangle$ is a left-invariant metric, for a orthonormal basis relative to $\langle,\rangle$, we have
\begin{equation}\label{killing}
  Ric=\lambda \mathrm{Id}-\frac{1}{2}[(\mathrm{ad} X)+(\mathrm{ad} X)^t]+\frac{|X|^2}{m}\mathrm{Pr}|_X.
\end{equation}
Here $\mathrm{Pr}|_X$ is the projection of a vector to $X$. As scalar curvature is a Riemannian invariant, $sc(\langle,\rangle)=sc(\phi^*_t\langle,\rangle)$, where $\phi_t=\exp t\mathrm{ad}X$ is a 1-parameter subgroup of $Aut(G)$. By Lemma~\ref{sc}, the equation~(\ref{killing}), and the fact that $G$ is unimodular.
\begin{eqnarray*}
  0&=&\frac{d}{dt}\mid_{t=0}sc(\phi^*_t\langle,\rangle)=(\mathrm{grad}\ sc,\mathrm{ad} X)_{\langle,\rangle} \\
              &=&-\lambda\mathrm{tr}\ \mathrm{ad}X+\mathrm{tr}\ \frac{1}{2}[(\mathrm{ad} X)+(\mathrm{ad} X)^t](\mathrm{ad}\ X)-\frac{|X|^2}{m}\mathrm{tr}\ (\mathrm{Pr}|_X) (\mathrm{ad}\ X)\\
             &=&-\lambda\mathrm{tr}\ \mathrm{ad}X+\mathrm{tr}\ \frac{1}{4}[(\mathrm{ad} X)+(\mathrm{ad} X)^t]^2-\frac{|X|^2}{m}\mathrm{tr}\ (\mathrm{ad}\ X)(\mathrm{Pr}|_X),\\
             &=&\mathrm{tr}\ \frac{1}{4}[(\mathrm{ad} X)+(\mathrm{ad} X)^t]^2.
\end{eqnarray*}
Here $\langle \mathrm{ad}X(Y),Z \rangle=\langle Y, (\mathrm{ad}X)^t(Z) \rangle$ for any $X,Y,Z\in{\mathfrak g}$. Thus we have $\frac{1}{2}[(\mathrm{ad} X)+(\mathrm{ad} X)^t]=0$ and hence $X$ is a Killing vector field. Then we have
\begin{theorem}\label{killingfield1}
Let $G$ be a unimodular Lie group with a left-invariant metric $\langle,\rangle$. If $X$ is a left-invariant vector field on $G$ such that $\langle,\rangle$ is $m$-quasi-Einstein, i.e., $Ric^m_X=\lambda\langle,\rangle$, then $X$ is a Killing vector field.
\end{theorem}

\begin{remark}
Jablonski proved in \cite{Jab1} that compact homogeneous Ricci solitons are necessarily Einstein. In essential, it is equivalent that the vector field is a Killing vector field. The proof there for Ricci solitons inspired the proof here for $m$ finite on unimodular Lie groups.
\end{remark}

Clearly, a compact Lie group is unimodular. By Theorems~\ref{leftinv} and \ref{killingfield1}, we have Theorem~\ref{killingfield}.

\section{$m$-quasi-Einstein metrics on compact simple Lie groups}
Let $G$ be a compact Lie group with Lie algebra ${\mathfrak g}$ and $g$ a bi-invariant metric on $G$. Let $K$ be a connected subgroup of $G$ with Lie algebra ${\mathfrak k}$. Thus ${\mathfrak k}$ is compact and splits into center and simple ideals: $${\mathfrak k}={\mathfrak k}_0\oplus{\mathfrak k}_1\oplus\cdots\oplus{\mathfrak k}_r,$$ where ${\mathfrak k}_0$ is the center of ${\mathfrak k}$. Clearly, this decomposition is orthogonal with respect to $g$. Let ${\mathfrak p}={\mathfrak k}^\bot$ with respect to $g$.
\begin{theorem}[\cite{DZ1}]
For any compact Lie group $G$ with a bi-invariant metric $g$ and a connected subgroup $K$, the following left-invariant metrics
\begin{equation}\label{NatRed}\langle,\rangle=ag|_{\mathfrak p}+h|_{\mathfrak k_0}+a_1g|_{{\mathfrak k}_1}+\cdots+a_rg|_{{\mathfrak k}_r}\end{equation}
are naturally reductive with respect to $G\times K$. Here ${\mathfrak k}={\mathfrak k}_0\oplus{\mathfrak k}_1\oplus\cdots\oplus{\mathfrak k}_r$, ${\mathfrak k}_0$ is the center of ${\mathfrak k}$, ${\mathfrak k}_1, \cdots,{\mathfrak k}_r$ are simple ideals of $\mathfrak k$, $a,a_1,\cdots,a_r$ are positive, and $h$ is an arbitrary metric on $\mathfrak k_0$.
\end{theorem}
Furthermore, let $G$ be a compact simple Lie group.
\begin{theorem}[\cite{DZ1}]
Any left-invariant metric on a compact simple Lie group which is naturally reductive is of the form~(\ref{NatRed}).
\end{theorem}

Without loss of generalization, assume that $V=\{X\in{\mathfrak p}|[X,\mathfrak k]=0\}=0$, otherwise we can enlarge the group $K$ such that $V=0$.

Assume that $G$ is a compact simple Lie group and $\dim {\mathfrak k}_0\leqslant 1$. Let $B,B_i$ be the Killing form of ${\mathfrak g}$, ${\mathfrak k}_i$. Then $h|_{\mathfrak k_0}=a_0(-B)|_{\mathfrak k_0}$, where $a_0>0$ if $\dim {\mathfrak k}_0=1$ and $a_0>0$ if $\dim {\mathfrak k}_0=0$. Since every ${\mathfrak k}_i$ for $1\leqslant i\leqslant r$ is simple, we have $B_i=c_iB$ with $c_i>0$ while $B_j=c_jB$ with $c_j=0$ for $j=0$.

Let $\pi$, $\pi_i$ be the projections of $\mathfrak g$ on $\mathfrak p$, $\mathfrak k_i$ respectively.
In order to describe the Ricci curvature, define the following symmetrical bilinear forms $A_i$, $i=0,1,\cdots, r$ and $T$ on $\mathfrak p$ as in \cite{DZ1} and \cite{Jen1}:
$$A_i(X,Y)=\mathrm{Tr}_{\mathfrak p}(\mathrm{ad} X)(\pi_i \mathrm{ad} Y),\ T(X,Y)=\mathrm{Tr}_{\mathfrak p}(\pi\mathrm{ad} X)(\pi \mathrm{ad} Y).$$
Here $\mathrm{Tr}_{\mathfrak p}$ denotes the trace of the mapping restricted on $\mathfrak p$, the others are similar. Moreover, it is shown in \cite{DZ1} that
$$A_i(X,Y)=\mathrm{Tr}_{\mathfrak p}(\mathrm{ad} X)(\pi_i \mathrm{ad} Y)=\mathrm{Tr}_{\mathfrak k_i}(\pi_i \mathrm{ad} Y)(\mathrm{ad} X).$$
It is easy to see that $A_i$ and $T$ are negative semidefinite and $Ad (K)$ invariant. Furthermore it is well known (see \cite{DZ1} and \cite{Jen1}) that
$$B|_{\mathfrak p}=T+2\sum_{i=0}^rA_i.$$
For any $X\in {\mathfrak g}$, define $A_XY=-\nabla_YX$, we have a simple formula for the Ricci curvature of a left-invariant metric (see \cite{DZ1} and \cite{Sa1}):
\begin{equation}\label{Ric}
Ric(Y,Z)=-\mathrm{Tr} A_ZA_Y, \forall Y,Z\in {\mathfrak g}.
\end{equation}
The Ricci curvature is determined by the equation~(\ref{Ric}).
\begin{lemma}[\cite{DZ1}]\label{Ricci}
Let $G$ be a compact simple Lie group $G$ with the bi-invariant metric $g=-B$, let $K$ be a connected subgroup of $G$ with $\dim {\mathfrak k}_0\leqslant 1$, and let
\begin{equation}\label{NatRed1}\langle,\rangle=ag|_{\mathfrak p}+a_0g|_{\mathfrak k_0}+a_1g|_{{\mathfrak k}_1}+\cdots+a_rg|_{{\mathfrak k}_r}\end{equation}
be a naturally reductive metric on $G$ with respect to $G\times K$. Then the Levi-Civita connection with respect to $\langle,\rangle$ is given by
\[ \left\{ \begin{aligned}
   & \nabla_XY=\frac{1}{2}[X,Y] \text{ for any } X,Y\in {\mathfrak k} \text{ or } X,Y\in {\mathfrak p}, \\
   & \nabla_XY=\frac{a_i}{2a}[X,Y] \text{ for any } X\in {\mathfrak p} \text{ and } Y\in {\mathfrak k_i}, \\
   & \nabla_XY=(1-\frac{a_i}{2a})[X,Y] \text{ for any } X\in {\mathfrak k_i} \text{ and } Y\in {\mathfrak p}.
\end{aligned}\right.\]
Furthermore, the Ricci curvature with respect to $\langle,\rangle$ is given by
\[ \left\{ \begin{aligned}
  & Ric|_{\mathfrak k_j}=-\frac{1}{4a^2}(a^2c_j-a_j^2c_j+a_j^2)B|_{\mathfrak k_j},\quad 0\leqslant j \leqslant r,\\
  & Ric|_{\mathfrak p}=\frac{1}{2}\sum_{i=0}^r(\frac{a_i}{a}-1)A_i-\frac{1}{4}B|_{\mathfrak p}, \\
  & Ric({\mathfrak p},{\mathfrak k}_i)=Ric({\mathfrak k}_i,{\mathfrak k}_j)=0,\quad 0\leqslant i\not=j \leqslant r.
\end{aligned}\right.\]
\end{lemma}

Let $s_i=\dim {\mathfrak k}_i$ and $n=\dim {\mathfrak p}$. Then it is showed in \cite{DZ1} (compare \cite{Jen1}, page 610) that \begin{equation}\label{A}\sum_jA_i(X_j,X_j)=-s_i(1-c_i),\quad i=0,1,\cdots,r\end{equation} for an orthonormal basis $X_j$ of $\mathfrak p$ with respect to $B$.

In the following, we will study the $m$-quasi-Einstein metrics among the metrics appearing in Lemma~\ref{Ricci}. By Theorem~\ref{killingfield}, the vector field $X$ is a left-invariant Killing vector field with respect to $\langle,\rangle$. Furthermore, we have:

\begin{theorem}\label{Quasi-E}
Let $G$, $K$, ${\mathfrak k}_i$, $\mathfrak p$ and $\langle,\rangle$ be those in Lemma~\ref{Ricci}. If $\langle,\rangle$ is $m$-quasi-Einstein, i.e., there exists $X\in {\mathfrak g}$ such that $Ric^m_X=\lambda \langle,\rangle$ for some $\lambda\in {\mathbb R}$, then $X\in {\mathfrak k}_0$. In particular if $\dim {\mathfrak k}_0=0$, then $\langle,\rangle$ is trivial.
\end{theorem}
\begin{proof}
Let ${\mathfrak p}={\mathfrak p_1}\oplus{\mathfrak p_2}\oplus\cdots\oplus{\mathfrak p_l}$, where ${\mathfrak p_i}$ is $\mathrm{ad} {\mathfrak k}$ irreducible. Since $V=0$, we have $\dim {\mathfrak p_i}>1$. Since $A_i$ and $B_{\mathfrak p}$ are $\mathrm{ad} {\mathfrak k}$ irreducible, $A_i|_{{\mathfrak p}_j}=a_{ij}B|_{{\mathfrak p}_j}$. The $a_{ij}$ are completely determined by the imbedding $\mathfrak k\subset \mathfrak g$, and we have $$Ric|_{\mathfrak p_j}=\frac{1}{2}\sum_{i=0}^r(\frac{a_i}{a}-1)a_{ij}B|_{\mathfrak p_j}-\frac{1}{4}B|_{\mathfrak p_j}.$$
Assume that $X=X_{\mathfrak k}+X_{\mathfrak p_1}+\cdots+X_{\mathfrak p_l}$. For any $Y,Z\in {\mathfrak p_j}$,
\begin{eqnarray*}
\mathfrak{L}_X\langle Y,Z\rangle&=&\langle \nabla_YX,Z\rangle+\langle Y,\nabla_ZX\rangle=\langle [Y,X],Z\rangle+\langle Y,[Z,X]\rangle \\
&=& \langle [Y,X_{\mathfrak k}],Z\rangle+\langle Y,[Z,X_{\mathfrak k}]\rangle=ag([Y,X_{\mathfrak k}],Z)+ag(Y,[Z,X_{\mathfrak k}]) \\
&=&ag(Y,[X_{\mathfrak k},Z])+ag(Y,[Z,X_{\mathfrak k}])=0.
\end{eqnarray*}
If $\langle,\rangle$ is $m$-quasi-Einstein,
\begin{eqnarray*}
Ric^m_X(Y,Z)&=&Ric(Y,Z)-\frac{1}{m}\langle X,Y\rangle\langle X,Z\rangle \\
            &=& \frac{1}{2}\sum_{i=0}^r(\frac{a_i}{a}-1)a_{ij}B(Y,Z)-\frac{1}{4}B(Y,Z)-\frac{1}{m}\langle X,Y\rangle\langle X,Z\rangle \\
             &=&\lambda\langle Y,Z\rangle\\
             &=&-\lambda aB(Y,Z).
\end{eqnarray*}
Let $\{e_1,\cdots,e_s\}$ be an orthonormal basis of $\mathfrak p_j$ with respect to $B$. Then
$$\langle X,e_i\rangle\langle X,e_j\rangle=\langle X_{\mathfrak p_j},e_i\rangle\langle  X_{\mathfrak p_j},e_j\rangle=\delta_{ij}p$$ for some constant $p$. Set $X_{\mathfrak p_j}=\sum_{i=1}^sn_ie_i$. Then we have $a^2n_in_j=\delta_{ij}p$. It implies that $X_{\mathfrak p_j}=0$ since $\dim {\mathfrak p_j}>1$. Since $j$ is arbitrary, we have that $X\in {\mathfrak k}$.

For any $1\leqslant i\leqslant r$, similarly, $\pi_i(X)=0$ by discussing the restriction of $Ric^m_X$ on $\mathfrak k_i$. That is, $X\in {\mathfrak k}_0$.
\end{proof}

Assume that ${\mathfrak k}$ acts irreducibly on $\mathfrak p$. Let $A_i|_{\mathfrak p}=b_iB|_{\mathfrak p}$. By the equation~(\ref{A}), $$b_i=\frac{s_i(1-c_i)}{n},\quad i=0,1,\cdots,r.$$ For this case, we have
$$Ric|_{\mathfrak p}=\{\sum_{i=0}^r\frac{(\frac{a_i}{a}-1)s_i(1-c_i)}{2n}-\frac{1}{4}\}B|_{\mathfrak p}.$$

Let $e_0$ be an orthonormal basis of $\mathfrak k_0$ with respect to $B$. In addition, we normalize the metric so that $a=1$. Then $\langle,\rangle$ is an $m$-quasi-Einstein metric, if and only if, $X=n_0e_0$ and the following equations hold:
\begin{eqnarray}
   &&a_0(m-4n_0^2)=4m\lambda, \label{1}\\
   &&(1-a_i^2)c_i+a_i^2=4\lambda a_i, \text{ for any } 1\leqslant i\leqslant r,\label{2}\\
   &&-\sum_{i=0}^r\frac{(a_i-1)s_i(1-c_i)}{2n}+\frac{1}{4}=\lambda. \label{3}
\end{eqnarray}
Putting the equation~(\ref{3}) back into the equations~(\ref{1}) and (\ref{2}), we have
\begin{eqnarray}
   &&a_0(m-4n_0^2)=m-\frac{2m}{n}\sum_{i=0}^r(a_i-1)s_i(1-c_i), \label{1'}\\
   &&(1-a_i^2)c_i+a_i^2=a_i-\frac{2a_i}{n}\sum_{j=0}^r(a_j-1)s_j(1-c_j), \text{ for any } 1\leqslant i\leqslant r.\label{2'}
\end{eqnarray}

Let $4n_0^2=(1-p)m$. Since $m>0$ and $4n_0^2\geqslant 0$, we have $p\leqslant 1$. Then $n_0=0$ if and only if $p=1$, and the equation~(\ref{1'}) is changed to
\begin{equation}\label{irr1}
 a_0p=1-\frac{2}{n}\sum_{i=0}^r(a_i-1)s_i(1-c_i).
\end{equation}
Putting the above equation back into the equation~(\ref{2'}), we have
\begin{equation}\label{irr2}
a_0a_ip=(1-a_i^2)c_i+a_i^2, \text{ for any } 1\leqslant i\leqslant r
\end{equation}

\begin{proposition}\label{thm1}
Let $G$, $K$, ${\mathfrak k}_i$, $\mathfrak p$ and $\langle,\rangle$ be as above. Assume that ${\mathfrak k}$ acts irreducibly on $\mathfrak p$. Then $\langle, \rangle$ is $m$-quasi-Einstein if and only if there exists $p\leqslant 1$ such that the equations~(\ref{irr1}) and (\ref{irr2}) hold. In particular, $\langle,\rangle$ is trivial if and only if $p=1$.
\end{proposition}
\begin{proof}
The proposition follows from the above discussion and the fact that the equation~(\ref{3}) only determines $\lambda$ in terms of $a_i$.
\end{proof}

If $r=0$, i.e., ${\mathfrak k}=\mathfrak k_0$, then ${\mathfrak g}={\mathfrak su}(2)$, $\mathfrak k={\mathfrak su}(1)$, and the equations~(\ref{irr1}) and~(\ref{irr2}) are just
\begin{equation}\label{one}
   (p+1)a_0=2.
\end{equation}
It follows that every $a_0$ satisfying $a_0=\frac{2}{p+1}$ for $-1<p<1$ gives a non-trivial $m$-quasi-Einstein metric on ${\mathfrak su}(2)$. Also we have $$\lambda=2-a_0=a_0p\leqslant 0$$ if $-1<p\leqslant 0$.

\begin{remark}
A gradient Ricci soliton on a compact manifold $M^n$ with
$\lambda\leqslant 0$ is trivial \cite{ELM1}. The same result was proved in \cite{KK1} for gradient $m$-quasi-Einstein metrics on compact manifolds with $m$ finite. The above discussion shows that not every $m$-quasi-Einstein metric on a compact manifold with
$\lambda\leqslant 0$ is trivial.
\end{remark}

\begin{remark} The well known Alekseevskii conjecture is that every homogeneous Einstein metric with negative
scalar curvature is isometric to a simply-connected solvmanifold. In \cite{Jab2}, Jablonski proves that it is equivalent to the generalized Alekseevskii conjecture: every expanding homogeneous Ricci soliton is isometric to a simply-connected solvmanifold. Here we can construct non-compact and non-solvable Lie groups admitting expanding $m$-quasi-Einstein metrics, which shows that the Alekseevskii conjecture for $m$-quasi-Einstein metrics fails. In fact, let $R$ be a solvable Lie group admitting an Einstein metric $g^R$ with the Einstein constant $\lambda<0$. Clearly $a_0=2-\lambda$ and $p=-1+\frac{2}{2-\lambda}$ determine an $m$-quasi-Einstein metric $g^S$ on ${\mathfrak su}(2)$ with the same $\lambda$. Define the metric $\langle,\rangle$ on ${\mathfrak su}(2)\oplus R$ by $$\langle,\rangle|_{{\mathfrak su}(2)}=g^S,\quad\langle,\rangle|_{R}=g^R,\quad\langle {\mathfrak su}(2) ,R\rangle=0.$$ Then $\langle,\rangle$ is a expanding $m$-quasi-Einstein metric.
\end{remark}

If $r=1$, i.e., ${\mathfrak k}=\mathfrak k_0\oplus {\mathfrak k}_1$, then the equations~(\ref{irr1}) and~(\ref{irr2}) are just
\begin{eqnarray}
  && (np+2)a_0=n+2-2(a_1-1)s_1(1-c_1), \label{F1} \\
  && a_0a_1p=(1-a_1^2)c_1+a_1^2. \label{F2}
\end{eqnarray}

\begin{proposition}\label{thm2}
Let $G$, $K$, ${\mathfrak k}_i$, $\mathfrak p$ and $\langle,\rangle$ be as above. Assume that ${\mathfrak k}=\mathfrak k_0\oplus {\mathfrak k}_1$ acts irreducibly on $\mathfrak p$. Then $\langle, \rangle$ is $m$-quasi-Einstein if and only if $a_0,a_1,p$ are determined by the following conditions:
\[ \left\{ \begin{aligned}
   & (2s_1+n+2)(1-c_1)a_1^2-(n+2+2s_1(1-c_1))a_1+(n+2)c_1\leqslant 0, \\
   &p=\frac{2(1-c_1)a_1^2+2c_1}{-(2s_1+n)(1-c_1)a_1^2+(n+2+2s_1(1-c_1))a_1-nc_1}, \\
   &a_0=\frac{-(2s_1+n)(1-c_1)a_1^2+(n+2+2s_1(1-c_1))a_1-nc_1}{2a_1}.
  \end{aligned} \right. \]
In particular, $\langle,\rangle$ is trivial if and only if $``="$ holds.
\end{proposition}
\begin{proof}
By Proposition~\ref{thm1}, $\langle, \rangle$ is $m$-quasi-Einstein if and only if the equations~(\ref{F1}) and (\ref{F2}) hold for $p\leqslant 1$. By the equations~(\ref{F1}) and (\ref{F2}), we have
\[ \left\{ \begin{aligned}
   &p=\frac{2(1-c_1)a_1^2+2c_1}{-(2s_1+n)(1-c_1)a_1^2+(n+2+2s_1(1-c_1))a_1-nc_1}, \\
   &a_0=\frac{-(2s_1+n)(1-c_1)a_1^2+(n+2+2s_1(1-c_1))a_1-nc_1}{2a_1}.
\end{aligned}\right.\]
 By the above discussion, we need to guarantee $p\leqslant 1$. First we must have
\begin{equation}\label{eee}-(2s_1+n)(1-c_1)a_1^2+(n+2+2s_1(1-c_1))a_1-nc_1>0.\end{equation} Then $p\leqslant 1$ if and only if
$$-(2s_1+n)(1-c_1)a_1^2+(n+2+2s_1(1-c_1))a_1-nc_1\geqslant 2(1-c_1)a_1^2+2c_1,$$ which implies (\ref{eee}) holds. Then the proposition follows.
\end{proof}

If $G$ is a compact simple Lie group and ${\mathfrak k}=\mathfrak k_0\oplus {\mathfrak k}_1$ acts irreducibly on $\mathfrak p$, then $G/K$ is symmetric, $\dim {\mathfrak k}_0=1$ and we have the following cases (for the details see \cite{DZ1}):
\begin{enumerate}
   \item ${\mathfrak u}(k)\subset {\mathfrak so}(2k)$ for $k>2$, $c_1=\frac{k}{2(k-1)}$, $s_1=k^2-1$ and $n=k(k-1);$
   \item ${\mathfrak u}(k)\subset {\mathfrak sp}(k)$ for $k>1$, $c_1=\frac{k}{2(k+1)}$, $s_1=k^2-1$ and $n=k(k+1);$
   \item ${\mathfrak s}({\mathfrak u}(1)\oplus{\mathfrak u}(k))\subset {\mathfrak su}(k+1)$ for $k>1$, $c_1=\frac{k}{k+1}$, $s_1=k^2-1$ and $n=2k;$
   \item ${\mathfrak so}(2)\oplus{\mathfrak so}(k)\subset {\mathfrak so}(k+2)$ for $k>2$, $c_1=\frac{k-2}{k}$, $s_1=\frac{k(k-1)}{2}$ and $n=2k;$
   \item ${\mathfrak so}(10)\oplus{\mathfrak so}(2)\subset {\mathfrak e}_6$, $c_1=\frac{2}{3}$, $s_1=45$ and $n=32;$
   \item ${\mathfrak e}_6\oplus{\mathfrak so}(2)\subset {\mathfrak e}_7$, $c_1=\frac{2}{3}$, $s_1=78$ and $n=54.$
\end{enumerate}

By Proposition~\ref{thm2}, in order to given non-trivial $m$-quasi-Einstein metrics for the above cases, we just need to find $a_1>0$ such that
\begin{equation*}
(2s_1+n+2)(1-c_1)a_1^2-(n+2+2s_1(1-c_1))a_1+(n+2)c_1< 0.
\end{equation*}
That is,
\begin{equation}\label{inequ}
((2s_1+n+2)(1-c_1)a_1-(n+2)c_1)(a_1-1)< 0.
\end{equation}

For case $(1)$,
$$\frac{(n+2)c_1}{(2s_1+n+2)(1-c_1)}=\frac{k^2-k+2}{3k^2-7k+2}<1 \text{ for } k\geqslant 4; =1 \text{ if } k=3.$$ It follows that every $a_1$ satisfying $\frac{k^2-k+2}{3k^2-7k+2}< a_1< 1$ gives a non-trivial $m$-quasi-Einstein metric on ${\mathfrak so}(2k)$ for $k\geqslant 4$.

For case $(2)$,
$$\frac{(n+2)c_1}{(2s_1+n+2)(1-c_1)}=\frac{k^2+k+2}{3k^2+7k+2}<1 \text{ for } k\geqslant 2.$$ It follows that every $a_1$ satisfying $\frac{k^2+k+2}{3k^2+7k+2}< a_1< 1$ gives a non-trivial $m$-quasi-Einstein metric on ${\mathfrak sp}(k)$ for $k\geqslant 2$.

For case $(3)$,
$$\frac{(n+2)c_1}{(2s_1+n+2)(1-c_1)}=1 \text{ for any } k.$$ Thus we give no non-trivial $m$-quasi-Einstein metrics for this case.

For case $(4)$,
$$\frac{(n+2)c_1}{(2s_1+n+2)(1-c_1)}=\frac{k^2-k-2}{k^2+k+2}<1 \text{ for } k\geqslant 3.$$ It follows that every $a_1$ satisfying $\frac{k^2-k-2}{k^2+k+2}< a_1< 1$ gives a non-trivial $m$-quasi-Einstein metric on ${\mathfrak so}(k+2)$ for $k\geqslant 3$.

For case $(5)$,
$$\frac{(n+2)c_1}{(2s_1+n+2)(1-c_1)}=\frac{17}{31}.$$ It follows that every $a_1$ satisfying $\frac{17}{31}< a_1< 1$ gives a non-trivial $m$-quasi-Einstein metric on ${\mathfrak e}_6$.

For case $(6)$,
$$\frac{(n+2)c_1}{(2s_1+n+2)(1-c_1)}=\frac{28}{53}.$$ It follows that every $a_1$ satisfying $\frac{28}{53}< a_1< 1$ gives a non-trivial $m$-quasi-Einstein metric on ${\mathfrak e}_7$.

If $r=2$, i.e., ${\mathfrak k}=\mathfrak k_0\oplus {\mathfrak k}_1\oplus {\mathfrak k}_2$, then the equations~(\ref{irr1}) and~(\ref{irr2}) are just
\begin{eqnarray}
  && (np+2)a_0=n+2-2(a_1-1)s_1(1-c_1)-2(a_2-1)s_2(1-c_2), \label{F21} \\
  && a_0a_ip=(1-a_i^2)c_i+a_i^2,i=1,2. \label{F22}
\end{eqnarray}

If ${\mathfrak k}=\mathfrak k_0\oplus {\mathfrak k}_1\oplus {\mathfrak k}_2$ with $\dim {\mathfrak k}_0\not=0$ acts irreducibly on $\mathfrak p$, then $\dim {\mathfrak k}_0\not=1$, and we have the following case:
\begin{enumerate}
  \item ${\mathfrak s}({\mathfrak u}(l_1)\oplus{\mathfrak u}(l_2))\subset {\mathfrak su}(l_1+l_2)$ for $l_1,l_2>1$, $c_1=\frac{l_1}{l_1+l_2}$, $c_2=\frac{l_2}{l_1+l_2}$, $s_1=l_1^2-1$, $s_2=l_2^2-1$ and $n=2l_1l_2$.
\end{enumerate}
For this case, the equations~(\ref{F21}) and~(\ref{F22}) are
\[ \left\{ \begin{aligned}
   & pa_0a_1=\frac{l_1}{l_1+l_2}+\frac{l_2}{l_1+l_2}a_1^2, \\
   & pa_0a_2=\frac{l_2}{l_1+l_2}+\frac{l_1}{l_1+l_2}a_2^2, \\
   & (pl_1l_2+1)a_0=l_1l_2+1-\frac{(a_1-1)(l_1^2-1)l_2}{l_1+l_2}-\frac{(a_2-1)(l_2^2-1)l_1}{l_1+l_2}.
\end{aligned} \right.\]
By the first two equations, we have $$(a_1l_2-a_2l_1)(1-a_1a_2)=0.$$
Let $a_1=\frac{l_1}{l_2}a_2$. Then we have:
\[ \left\{ \begin{aligned}
   & p=\frac{l_2+l_1a_2^2}{-(l_1^3+l_1^2l_2+l_1l_2^2-2l_1)a_2^2+2l_1l_2(l_1+l_2)a_2-l_1l_2^2}, \\
   & a_0=\frac{-(l_1^3+l_1^2l_2+l_1l_2^2-2l_1)a_2^2+2l_1l_2(l_1+l_2)a_2-l_1l_2^2}{a_2(l_1+l_2)}.
\end{aligned}\right.\]

\begin{proposition}\label{su}
Let $\mathfrak g={\mathfrak su}(l_1+l_2)$ and ${\mathfrak k}={\mathfrak s}({\mathfrak u}(l_1)\oplus{\mathfrak u}(l_2))$, where $l_1,l_2>1$. Then $\langle,\rangle$ is $m$-quasi-Einstein if $p,a_0,a_1,a_2$ satisfy the following conditions:
\[ \left\{ \begin{aligned}
   & p=\frac{l_2+l_1a_2^2}{-(l_1^3+l_1^2l_2+l_1l_2^2-2l_1)a_2^2+2l_1l_2(l_1+l_2)a_2-l_1l_2^2}, \\
   & a_0=\frac{-(l_1^3+l_1^2l_2+l_1l_2^2-2l_1)a_2^2+2l_1l_2(l_1+l_2)a_2-l_1l_2^2}{a_2(l_1+l_2)}, \\
   & a_1=\frac{l_1}{l_2}a_2, \\
   &  x_1\leqslant a_2\leqslant x_2,
\end{aligned} \right.\]
where $x_1<x_2$ are the distinct positive real solutions of the equation $$(l_1^3+l_1^2l_2+l_1l_2^2-l_1)a_2^2-2l_1l_2(l_1+l_2)a_2+(l_1l_2+1)l_2=0.$$ In particular, $\langle,\rangle$ is trivial if and only if $a_2=x_1$ or $a_2=x_2$.
\end{proposition}
\begin{proof}
In order to give $m$-quasi-Einstein metrics, it is enough to find $a_2>0$ satisfying $$-(l_1^3+l_1^2l_2+l_1l_2^2-2l_1)a_2^2+2l_1l_2(l_1+l_2)a_2-l_1l_2^2\geqslant l_2+l_1a_2^2.$$ That is,
\begin{equation}\label{inequ2}(l_1^3+l_1^2l_2+l_1l_2^2-l_1)a_2^2-2l_1l_2(l_1+l_2)a_2+(l_1l_2+1)l_2\leqslant 0.\end{equation}
It is easy to check that
$$(2l_1l_2(l_1+l_2))^2-4(l_1^3+l_1^2l_2+l_1l_2^2-l_1)(l_1l_2+1)l_2=4l_1l_2(l_1^2-1)(l_2^2-1)>0.$$
Then $x_1\leqslant a_2\leqslant x_2$, where $x_1<x_2$ are the distinct positive real solutions of the equation $(l_1^3+l_1^2l_2+l_1l_2^2-l_1)a_2^2-2l_1l_2(l_1+l_2)a_2+(l_1l_2+1)l_2=0$. When $a_2=t_1$ or $a_2=t_2$, we have that $p=1$, which corresponds to a trivial $m$-quasi-Einstein metric.
\end{proof}

In summary, we have:

\begin{theorem}\label{1.1}
Every compact simple Lie group admits infinitely many non-trivial $m$-quasi-Einstein metrics except $SU(3)$, $E_8$, $F_4$ and $G_2$.
\end{theorem}

Assume that ${\mathfrak k}$ doesn't act irreducibly on $\mathfrak p$. The following is to study the case in Table $(I)$ from \cite{AMS1}.

\begin{table}[htb]
Table (I)\\
\begin{tabular}{|c|c|c|c|c|}
\hline
${\mathfrak g}$ & Diagram & $\dim {\mathfrak k}_1$ & $\dim {\mathfrak p}_1$ & $\dim {\mathfrak p}_2$ \\
\hline ${\mathfrak f}_4$ & \setlength{\unitlength}{0.7mm}
\begin{picture}(50,10)(0,0)
\put(0,2){\circle{2}}\put(0,2){\circle{0.6}}\put(40,2){\circle*{2}}\multiput(10,2)(10,0){3}{\circle{2}}
\multiput(1,2)(10,0){2}{\line(1,0){8}}
\multiput(21,1.2)(0,1.5){2}{\line(1,0){6.5}}\put(26,0.5){$>$}
\put(31,2){\line(1,0){8}}
\end{picture}
&  21 & 16  &  14  \\
\hline
\end{tabular}
\end{table}

For this case, ${\mathfrak g}={\mathfrak k_0}\oplus{\mathfrak k_1}\oplus{\mathfrak p_1}\oplus{\mathfrak p_2}$ and $\dim \mathfrak k_0=1$. Let $d_1=\dim {\mathfrak k_1}$,  $d_2=\dim {\mathfrak p_1}$ and $d_3=\dim {\mathfrak p_2}$. The naturally reductive metric $\langle,\rangle$ is given by
\begin{equation*}\langle,\rangle=a_0g|_{\mathfrak k_0}+a_1g|_{{\mathfrak k}_1}+ag|_{\mathfrak p}.\end{equation*}
The Ricci curvature corresponds to the metric is given in \cite{AMS1} by
\[ \left\{ \begin{aligned}
  &
  r_{{\mathfrak h}_0}=\frac{a_0}{4a^2}\frac{d_2}{(d_2+4d_3)}+\frac{a_0}{a^2}\frac{d_3}{(d_2+4d_3)},\\
  &
  r_{{\mathfrak
  h}_1}=\frac{1}{4d_1a_1}\frac{d_3(2d_1+2-d_3)}{(d_2+4d_3)}+\frac{a_1d_2}{4a^2(d_2+4d_3)}+
  \frac{a_1}{2d_1a^2}\frac{d_3(d_3-2)}{(d_2+4d_3)},\\
  &
  r_{{\mathfrak
  m}_1}=\frac{1}{2a}-\frac{1}{2a}\frac{d_3}{(d_2+4d_3)}-
  \frac{1}{2a^2}(a_0\frac{1}{(d_2+4d_3)}+a_1\frac{d_1}{d_2+4d_3}),\\
  &
  r_{{\mathfrak
  m}_2}=\frac{1}{a}\frac{2d_3}{(d_2+4d_3)}+\frac{1}{4a}\frac{d_2}{(d_2+4d_3)}-
  \frac{a_0}{a^2}\frac{2}{(d_2+4d_3)}-\frac{a_1(d_3-2)}{a^2(d_2+4d_3)}.
\end{aligned} \right. \]
The Ricci curvature is based on an orthonormal basis corresponding to $\langle,\rangle$. Let $e_0$ be a basis of ${\mathfrak k}_0$ satisfying $\langle e_0,e_0\rangle=1$. In addition, we normalize the metric so that $a=1$. Then $Ric^m_X=\lambda\langle,\rangle$, i.e., $\langle,\rangle$ is $m$-quasi-Einstein if and only if $X=n_0e_0$ and the following equations hold:
\[ \left\{ \begin{aligned}
  &
  \frac{a_0d_2}{4(d_2+4d_3)}+\frac{a_0d_3}{d_2+4d_3}-\frac{n_0^2}{m}=\lambda,\\
  &
  \frac{1}{2d_1a_1}\frac{d_3(2d_1+2-d_3)}{(d_2+4d_3)}+\frac{a_1d_2}{4(d_2+4d_3)}+
  \frac{a_1}{2d_1}\frac{d_3(d_3-2)}{(d_2+4d_3)}=\lambda,\\
  &
  \frac{1}{2}-\frac{d_3}{2(d_2+4d_3)}-\frac{a_0}{2(d_2+4d_3)}-\frac{a_1d_1}{2(d_2+4d_3)}=\lambda,\\
  &
  \frac{2d_3}{d_2+4d_3}+\frac{d_2}{4(d_2+4d_3)}-
  \frac{2a_0}{d_2+4d_3}-\frac{a_1(d_3-2)}{d_2+4d_3}=\lambda.
\end{aligned} \right. \]
Let $4n_0^2=pm$ since $n_0\geqslant 0,m>0$. Then $p\geqslant 0$. Putting the fourth equation back into the first three equations, we have
\[ \left\{ \begin{aligned}
  &
 (d_2+4d_3+8)a_0+4(d_3-2)a_1=p(d_2+4d_3)+d_2+8d_3,\\
  & a_0=-\frac{d_1d_2+(d_3-2)(4d_1+2d_3)}{8d_1}a_1-\frac{d_3(2d_1+2-d_3)}{4a_1d_1}+\frac{d_2+8d_3}{8},\\
  & (2d_1-4d_3+8)a_1-6a_0=d_2-2d_3.
\end{aligned} \right. \]

For the case, the equations are
\[ \left\{ \begin{aligned}
  &
 5a_0+3a_1=4p+8,\\
  & a_0=-10a_1-\frac{5}{a_1}+16,\\
  & a_0+a_1=2.
\end{aligned} \right. \]
It follows that $a_0=a_1=1,p=0$ or $a_0=\frac{5}{9},a_1=\frac{13}{9},p=\frac{2}{9}$. The second one corresponds to a non-trivial $m$-quasi-Einstein metric on ${\mathfrak f}_4$. That is,

\begin{theorem}\label{1.2}
There is a non-trivial $m$-quasi-Einstein metric on $F_4$.
\end{theorem}

{\noindent \it The proof of Theorem~\ref{main}:} Theorem~\ref{main} follows from Theorems~\ref{1.1} and \ref{1.2}.

\section{$m$-quasi-Einstein pseudo-Riemannian metrics on compact simple Lie groups}
It is natural to study $m$-quasi-Einstein pseudo-Riemannian metrics since the fundamental theorem in geometry holds for pseudo-Riemannian manifolds. This section is to find non-trivial $m$-quasi-Einstein pseudo-Riemannian metrics on compact simple Lie groups. In particular, we prove that every compact simple Lie group admits a non-trivial $m$-quasi-Einstein Lorentzian metric, which is still unknown for the trivial case. \

Let $G$ be a compact Lie group with Lie algebra ${\mathfrak g}$ and $g=-B$ the bi-invariant metric on $G$. Let $K$ be a connected subgroup of $G$ with Lie algebra ${\mathfrak k}$. Thus ${\mathfrak k}={\mathfrak k}_0\oplus{\mathfrak k}_1\oplus\cdots\oplus{\mathfrak k}_r,$ where ${\mathfrak k}_0$ is the center of ${\mathfrak k}$ with dimension $1$ and every $\mathfrak k_i$ for $i=1,2,\cdots,r$ is a simple ideal of $\mathfrak k_i$. Let ${\mathfrak p}={\mathfrak k}^\bot$ with respect to $g$.
Consider the following left-invariant pseudo-Riemannian metric on $G$
\begin{equation}\label{L1}\langle,\rangle=ag|_{\mathfrak p}+a_0g|_{\mathfrak k_0}+a_1g|_{{\mathfrak k}_1}+\cdots+a_rg|_{{\mathfrak k}_r},\end{equation}
where $a\not=0,a_0\not=0,a_1\not=0,\cdots a_r\not=0$. It is easy to check that the formulae of Levi-civita connections and Ricci curvatures with respect to the metric~(\ref{L1}) are the same as those in Lemma~\ref{Ricci}.

Assume that ${\mathfrak k}$ acts irreducibly on $\mathfrak p$. If $\langle,\rangle$ is $m$-quasi-Einstein, i.e., there exists $X\in {\mathfrak g}$ such that $Ric^m_X=\lambda \langle,\rangle$ for some $\lambda\in {\mathbb R}$, then $X\in {\mathfrak k}_0$, similar to the proof of Theorem~\ref{Quasi-E}. In addition, we normalize the metric so that $a=1$. Then $\langle,\rangle$ is $m$-quasi-Einstein if and only if the equations~(\ref{irr1}) and (\ref{irr2}) hold.

For the case when ${\mathfrak k}=\mathfrak k_0$ acts irreducibly on $\mathfrak p$. Then ${\mathfrak g}={\mathfrak su}(2)$, $\mathfrak k={\mathfrak su}(1)$, and the equations~(\ref{irr1}) and~(\ref{irr2}) are just
\begin{equation}
   (p+1)a_0=2.
\end{equation}
It follows that every $a_0$ satisfying $a_0=\frac{2}{p+1}$ for $p<-1$ gives a non-trivial $m$-quasi-Einstein Lorentzian metric on ${\mathfrak su}(2)$.

For the case when ${\mathfrak k}=\mathfrak k_0\oplus {\mathfrak k}_1$ acts irreducibly on $\mathfrak p$, the equations~(\ref{irr1}) and (\ref{irr2}) are just the equations~(\ref{F1}) and (\ref{F2}). We also have
\[ \left\{ \begin{aligned}
   &a_0=\frac{-(2s_1+n)(1-c_1)a_1^2+(n+2+2s_1(1-c_1))a_1-nc_1}{2a_1}, \\
   &p=\frac{2(1-c_1)a_1^2+2c_1}{-(2s_1+n)(1-c_1)a_1^2+(n+2+2s_1(1-c_1))a_1-nc_1}.
\end{aligned}\right.\]

If $a_1<0$, then $p<0$ and $a_0>0$. It follows that every $a_1<0$ gives a non-trivial $m$-quasi-Einstein pseudo-Riemannian metric on the compact simple Lie group with the signature $(1+\dim{\mathfrak p},\dim {\mathfrak k_1})$ and $\lambda<0$.

If $a_1>0$ and $a_0>0$, then $\langle,\rangle$ is $m$-quasi-Einstein if and only if $$\frac{(n+2)c_1}{(2s_1+n+2)(1-c_1)}\leqslant a_1\leqslant 1,$$ which is given in the previous section.

If $a_1>0$ and $a_0<0$, then we have
\begin{equation*}
(2s_1+n)(1-c_1)a_1^2-(n+2+2s_1(1-c_1))a_1+nc_1>0.
\end{equation*}
It follows that $p<0$. Since $(n+2+2s_1(1-c_1))^2-4(2s_1+n)(1-c_1)nc_1=(2s_1+n)(1-c_1)^2+(2+nc_1)^2+8(2s_1+n)(1-c_1)>0$, we know that the equation $$(2s_1+n)(1-c_1)a_1^2-(n+2+2s_1(1-c_1))a_1+nc_1=0$$ has two distinct positive real solutions $x_1<x_2$. Here $x_1<\frac{(n+2)c_1}{(2s_1+n+2)(1-c_1)}\leqslant 1<x_2$. Then $\langle,\rangle$ is $m$-quasi-Einstein if and only if $$a_1>x_2, \text{ or } x_1>a_1>0,$$ which gives a non-trivial $m$-quasi-Einstein Lorentzian metric on the compact simple Lie group with $\lambda>0$. That is,

\begin{theorem}\label{2.1}
The compact simple Lie groups $SU(n)$, $SO(n)$, $SP(n)$ and $E_6$ admit infinitely many non-trivial $m$-quasi-Einstein Lorentzian metrics.
\end{theorem}

Assume that ${\mathfrak k}$ doesn't act irreducibly on $\mathfrak p$. The following is to study the cases in Table $(II)$ from \cite{AMS1}.

\begin{table}[htb]
Table (II)\\
\begin{tabular}{|c|c|c|c|c|}
\hline
${\mathfrak g}$ & Diagram & $\dim {\mathfrak k}_1$ & $\dim {\mathfrak p}_1$ & $\dim {\mathfrak p}_2$ \tabularnewline
\hline \put(5,5){$\mathfrak e_8$} & \setlength{\unitlength}{0.7mm}
\begin{picture}(75,15)(0,0)
\multiput(0,2)(10,0){8}{\circle{2}}
\put(20,12){\circle{2}} \multiput(1,2)(10,0){7}{\line(1,0){8}}
\put(20,3){\line(0,1){8}} \put(60,2){\circle*{2}}
\put(70,2){\circle{0.6}}
\end{picture}
&  133 & 112  & 2  \tabularnewline
\hline $\mathfrak f_4$ & \setlength{\unitlength}{0.7mm}
\begin{picture}(50,10)(0,0)
\put(0,2){\circle{2}}\put(0,2){\circle{0.6}}\put(10,2){\circle*{2}}\multiput(20,2)(10,0){3}{\circle{2}}
\multiput(1,2)(10,0){2}{\line(1,0){8}}
\multiput(21,1.2)(0,1.5){2}{\line(1,0){6.5}}\put(26,0.5){$>$}
\put(31,2){\line(1,0){8}}
\end{picture}
&  21 & 28  &  2  \tabularnewline
\hline $\mathfrak g_2$ & \setlength{\unitlength}{0.7mm}
\begin{picture}(30,10)(0,0)
\ \put(0,2){\circle{2}}\put(0,2){\circle{0.6}}\put(10,2){\circle*{2}}\put(20,2){\circle{2}}
\multiput(1,2)(10,0){2}{\line(1,0){8}}
\multiput(10.2,1.2)(0,1.5){2}{\line(1,0){8}}\put(16,0.5){$>$}
\end{picture}
&  3 & 8  &  2  \tabularnewline
\hline
\end{tabular}
\end{table}

For these cases, ${\mathfrak g}={\mathfrak k_0}\oplus{\mathfrak k_1}\oplus{\mathfrak p_1}\oplus{\mathfrak p_2}$ and $\dim \mathfrak k_0=1$. Let $d_1=\dim {\mathfrak k_1}$ and $d_1=\dim {\mathfrak p_1}$. Consider the following Lorentzian metric $\langle,\rangle$ given by
\begin{equation*}\langle,\rangle=a_0g|_{\mathfrak k_0}+a_1g|_{{\mathfrak k}_1}+ag|_{\mathfrak p},\end{equation*}
where $a_0<0,a>0,a_1>0$. Following the computation of the Ricci curvature given in \cite{AMS1}, we have
\[ \left\{ \begin{aligned}
  &
  r_{{\mathfrak h}_0}=-\frac{a_0}{4a^2}\frac{d_2}{(d_2+8)}-\frac{a_0}{a^2}\frac{2}{(d_2+8)},\\
  &
  r_{{\mathfrak
  h}_1}=\frac{1}{4d_1a_1}(d_1-\frac{d_2(d_2+2)}{2(d_2+8)})+
  \frac{a_1}{4d_1a^2}\frac{d_2(d_2+2)}{2(d_2+8)},\\
  &
  r_{{\mathfrak
  m}_1}=\frac{1}{2a}-\frac{1}{2a}\frac{d_3}{(d_2+8)}-
  \frac{1}{2a^2}(a_0\frac{1}{(d_2+8)}+a_1\frac{d_2+2}{2(d_2+8)}),\\
  &
  r_{{\mathfrak
  m}_2}=\frac{1}{a}\frac{4}{(d_2+8)}+\frac{1}{4a}\frac{d_2}{(d_2+8)}-
  \frac{a_0}{a^2}\frac{2}{(d_2+8)}.
\end{aligned} \right. \]
The Ricci curvature is based on an orthonormal basis corresponding to $\langle,\rangle$. Let $e_0$ be the basis of ${\mathfrak k}_0$ satisfying $\langle e_0,e_0\rangle=-1$. In addition, we normalize the metric so that $a=1$. Then $Ric^m_X=\lambda\langle,\rangle$, i.e., $\langle,\rangle$ is $m$-quasi-Einstein if and only if $X=n_0e_0$ and the following equations hold:
\[ \left\{ \begin{aligned}
  &
  \frac{a_0d_2}{4(d_2+8)}+\frac{2a_0}{d_2+8}+\frac{n_0^2}{m}=\lambda,\\
  &
  \frac{1}{4d_1a_1}(d_1-\frac{d_2(d_2+2)}{2(d_2+8)})+
  \frac{a_1}{4d_1}\frac{d_2(d_2+2)}{2(d_2+8)}=\lambda,\\
  &
  \frac{1}{2}-\frac{d_3}{2(d_2+8)}-\frac{a_0}{2(d_2+8)}-\frac{a_1(d_2+2)}{4(d_2+8)})=\lambda,\\
  &
  \frac{4}{d_2+8}+\frac{d_2}{4(d_2+8)}-
  \frac{2a_0}{d_2+8}=\lambda.
\end{aligned} \right. \]
Let $4n_0^2=pm$ since $n_0\geqslant 0,m>0$. Then $p\geqslant 0$. Putting the fourth equation back into the first three equations, we have
\[ \left\{ \begin{aligned}
  &
 a_0+\frac{d_2+8}{d_2+16}p=1,\\
  & a_0=-\frac{d_2(d_2+2)}{16d_1}a_1+\frac{d_2(d_2+2)-2d_1(d_2+8)}{16a_1d_1}+\frac{d_2+16}{8},\\
  & 6a_0-(d_2+2)a_1+d_2-4=0.
\end{aligned} \right. \]

For any case in Table $(II)$, we have two solutions of the above equations. One is $a_0=a_1=1,p=0$, which isn't a Lorenztian metric; another is
\begin{enumerate}
  \item $a_0=-\frac{279}{25},a_1=\frac{9}{25},p=\frac{4864}{375}$ for ${\mathfrak e}_8$,
  \item $a_0=-\frac{8}{3},a_1=\frac{4}{15},p=\frac{121}{27}$ for ${\mathfrak f}_4$,
  \item $a_0=-\frac{1}{2},a_1=\frac{1}{10},p=\frac{9}{4}$ for ${\mathfrak g}_2$
\end{enumerate} respectively, which gives a non-trivial $m$-quasi-Einstein Lorentzian metric. That is,
\begin{theorem}\label{2.2}
There are non-trivial $m$-quasi-Einstein Lorentzian metrics on $E_8$, $F_4$ and $G_2$.
\end{theorem}

Theorem~\ref{main2} follows from Theorem~\ref{2.1} and \ref{2.2}.

\section{$m$-quasi-Einstein pseudo-Riemannian metrics on non-compact simple Lie groups}

Let $G$ be a compact simple Lie group with the Lie algebra $\mathfrak g$ and let $\theta$ be an involution of $\mathfrak g$. Let $${\mathfrak k}=\{x\in {\mathfrak g}|\theta(x)=x\} \text{ and }{\mathfrak p}=\{x\in {\mathfrak g}|\theta(x)=-x\}.$$
Then ${\mathfrak k}$ acts irreducibly on $\mathfrak p$, and we have ${\mathfrak g}={\mathfrak k_0}\oplus{\mathfrak k_1}\oplus\cdots\oplus{\mathfrak k_r}\oplus{\mathfrak p}$, where the dimension of the center ${\mathfrak k_0}$ of ${\mathfrak k}$ is no more than $1$. Clearly ${\mathfrak g}^{R}={\mathfrak k}\oplus i{\mathfrak p}$ is a non-compact simple Lie group. Let $G^R$ be the non-compact simple Lie group with the Lie algebra $\mathfrak g^R$.

In the following we only discuss the case when ${\mathfrak k_0}=1$. Let $g=-B$ be a bi-invariant metric on $G$. For the following left-invariant pseudo-Riemannian metric $\langle,\rangle$ on $G$ defined by
\begin{equation*}\langle,\rangle=ag|_{\mathfrak p}+a_0g|_{\mathfrak k_0}+a_1g|_{{\mathfrak k}_1}+\cdots+a_rg|_{{\mathfrak k}_r},\end{equation*}
we can define a pseudo-Rimennian metric $\langle,\rangle^R$ on $G^{R}$ by
\[ \left\{ \begin{aligned}
  & \langle iX,iY\rangle^R=-\langle X,Y\rangle,\quad  \forall X,Y\in {\mathfrak p}, \\
  & \langle X,Y\rangle^R=\langle X,Y\rangle, \quad \forall X,Y\in {\mathfrak k},\\
  & \langle {\mathfrak k_i},{\mathfrak p}\rangle^R=\langle {\mathfrak k_i},{\mathfrak k_j}\rangle^R=0,\quad \forall 0\leqslant i\not=j\leqslant r.
\end{aligned} \right. \]

\begin{proposition}\label{nonc}
Let the notations be as above. If $\langle,\rangle$ is $m$-quasi-Einstein, i.e., there exists $X\in {\mathfrak g}$ such that $Ric^m_X=\lambda \langle,\rangle$ for some $\lambda\in {\mathbb R}$, then $\langle,\rangle^R$ is $m$-quasi-Einstein with the vector $X$ and the constant $\lambda$.
\end{proposition}
\begin{proof}
Let ${\mathfrak g}^C$, ${\mathfrak k_0}^C$, ${\mathfrak k_1}^C$ and ${\mathfrak p}^C$
be the complexications of ${\mathfrak g}$, ${\mathfrak k}_0$, ${\mathfrak k}_1$ and ${\mathfrak p}$ respectively. Then the pseudo-Riemannian metric $\langle,\rangle$ can be naturally extended to a non-degenerate symmetric bilinear form $\langle,\rangle^C$ on ${\mathfrak g}^C$ such that
$$\langle,\rangle^C|_{\mathfrak g\times \mathfrak g}=\langle,\rangle \text{ and } \langle,\rangle^C|_{\mathfrak g^R\times \mathfrak g^R}=\langle,\rangle^R.$$ Then we can compute $^C\nabla_YZ$ and $^CRic_X^m(Y,Z)$ for any $Y,Z\in {\mathfrak g}^C$. Then the metric $\langle,\rangle$ is $m$-quasi-Einstein, i.e., $Ric_X^m(Y,Z)=\lambda\langle Y,Z\rangle$ for any $Y,Z\in {\mathfrak g}$, then $X\in {\mathfrak k_0}$. Furthermore we have $^CRic_X^m(Y,Z)=\lambda\langle Y,Z\rangle$ for any $Y,Z\in {\mathfrak g}^C$, which implies $^RRic_X^m(Y,Z)={^C}Ric_X^m(Y,Z)=\lambda\langle Y,Z\rangle$ for any $Y,Z\in {\mathfrak g}^R$. That is, $\langle,\rangle^R$ is $m$-quasi-Einstein with the same $\lambda$.
\end{proof}

In particular, for the case when ${\mathfrak k}=\mathfrak k_0\oplus {\mathfrak k}_1$ acts irreducibly on $\mathfrak p$, we obtain in the previous section that every $a_1<0$ gives a non-trivial $m$-quasi-Einstein pseudo-Riemannian metric on the compact simple Lie group with the signature $(1+\dim{\mathfrak p},\dim {\mathfrak k_1})$ and $\lambda<0$. Then by Proposition~\ref{nonc}, each of them gives an $m$-quasi-Einstein pseudo-Riemannian metric on $G^R$ with the signature $(1,\dim{\mathfrak p}+\dim {\mathfrak k_1})$.

For the case when ${\mathfrak k}=\mathfrak k_0\oplus {\mathfrak k}_1\oplus {\mathfrak k_2}$ acts irreducibly on $\mathfrak p$ and $a=1$, every $a_2<0$ gives a non-trivial $m$-quasi-Einstein pseudo-Riemannian metric on ${\mathfrak su}(l_1+l_2)$ with the signature $(1+2l_1l_2,l_1^2+l_2^2-2)$. Then by Proposition~\ref{nonc}, each of them gives an $m$-quasi-Einstein pseudo-Riemannian metric on ${\mathfrak su}(l_1,l_2)$ with the signature $(1,(l_1+l_2)^2-2)$.

For the above cases, every metric is non-trivial. By taking the negative of the pseudo-Riemannian metric, we have Theorem~\ref{main3}.

\section{Acknowledgments}
This work is supported by NSFC (No. 11001133). We would like to thank S.Q. Deng and E. Ribeiro Jr for the helpful comments, conversation and suggestions.

\end{document}